\documentclass[reqno,12pt]{article}
\usepackage[utf8]{inputenc}
\usepackage[margin=1in]{geometry}

\usepackage{amsmath}
\usepackage{amsfonts}
\usepackage{amssymb}
\usepackage{amsthm}
\usepackage{enumerate}
\usepackage{xcolor}

\newtheorem{theorem}{Theorem} 
\newtheorem{lemma}[theorem]{Lemma}
\newtheorem{proposition}[theorem]{Proposition}
\newtheorem{corollary}[theorem]{Corollary}

\usepackage{tikz}
\usepackage{pgfplots}
\pgfplotsset{compat=1.8}
\usetikzlibrary{datavisualization}
\pgfplotsset{every axis/.append style={
		axis x line=middle,    
		axis y line=middle,    
		axis line style={<->,color=black}, 
		xlabel={$x$},          
		ylabel={$y$},          
}}

\theoremstyle{definition}
\newtheorem{remark}[theorem]{Remark}
\newtheorem{definition}[theorem]{Definition}
\newtheorem{example}[theorem]{Example}

\newcommand{\real}{\mathbb{R}}

\newcommand{\norm}[1]{\left\| #1\right\|}

\title{Further results on angular equivalence of norms}
\author{Eder Kikianty}
\date{}

\begin{document}

\maketitle

\begin{abstract}
Angular equivalence of norms is introduced by Kikianty and Sinnamon (2017) and is a stronger notion than the usual topological equivalence. Given two angularly equivalent norms, if one norm has a certain geometrical property, e.g. uniform convexity, then the other norm also possesses such a property. In this paper, we show further results in this direction, namely angular equivalent norms share the property of uniform non-squareness, and that angular equivalence preserves the exposed points of the unit ball. A discussion on the (equivalence of the) dual norms of angularly equivalent norms is also given, giving a partial answer to an open problem as stated in the paper by Kikianty and Sinnamon (2017). 
\end{abstract}

\section{Introduction}

In the paper \cite{angular}, a new notion of norm equivalence, namely angular equivalence, is introduced. Two norms are angularly equivalent on a real vector space, if over all pairs of nonzero vectors, the angle of the pair with respect to one norm is comparable to the angle of the same pair with respect to the other norm. Any two norms that are angularly equivalent are also topologically equivalent. Angular equivalence preserves certain properties, e.g. uniform convexity, that the usual equivalence does not. 

\medskip

One needs a concept of angle in normed space to define such an equivalence. In a real normed space $(X, \norm{\cdot})$, the mapping $g^\pm \colon X\times X\to \real$ given by 
$$g^\pm(x,y):=\|x\|\lim_{t\rightarrow 0^\pm} \frac1t\left(\|x+ty\|-\|x\|\right)$$
exists. The $g$-functional relative to $\norm{\cdot}$ is defined as the map $g\colon X\times X\to \real$ given by
$$g(x,y):=\frac12(g^+(x,y)+g^-(x,y)), \quad x,y\in X. $$
We note that $g$ is not symmetric in general. If $x$ and $y$ are non-zero vectors in $X$, the norm angle from $x$ to $y$ is $\theta=\theta(x,y)$, defined by $0\leq \theta\leq \pi$ and
$$\cos\theta(x,y)=\frac{g(x,y)}{\|x\|\|y\|}.$$

\medskip

\noindent With this norm angle, angular equivalence is defined as follows. 
\begin{definition}[Kikianty and Sinnamon \cite{angular}]
Two norms $\norm{\cdot}_1$ and  $\norm{\cdot}_2$, on a real vector space $X$ are angularly equivalent provided there exists a constant $C$ such that for all non-zero $x,y\in X$ 
$$\tan\left(\frac{\theta_2(x,y)}{2}\right)\leq C \tan\left(\frac{\theta_1(x,y)}{2}\right).$$ 
Here $\theta_1(x,y)$ and $\theta_2(x,y)$ are the norm angles from $x$ to $y$ relative to $\norm{\cdot}_1$ and $\norm{\cdot}_2$, respectively. Also $\tan(\pi/2)$ is taken to be $+\infty$. 
\end{definition}
\noindent It is straightforward to see that angular equivalence is both reflexive and transitive. Despite appearances, angular equivalence is a symmetric relation (cf. \cite[p. 944]{angular}) and thus it is an equivalence relation. In what follows, we recall some results concerning angular equivalence, specifically the preservation of geometrical properties by this equivalence. For further results, we refer the readers to the paper \cite{angular}.

\begin{proposition}[Kikianty and Sinnamon \cite{angular}]\label{prop:results}
Let $\norm{\cdot}_1$ and $\norm{\cdot}_2$  be two angularly equivalent norms on the real vector space $X$. Then, the following statements are true. 
\begin{enumerate}[$(\mathrm{AE}1)$]
    \item Both norms $\norm{\cdot}_1$ and $\norm{\cdot}_2$ are topologically equivalent. 
    \item The norm $\norm{\cdot}_1$ is induced by an inner product if and only if $\norm{\cdot}_2$ is induced by an inner product.
    \item For $0\neq x\in X$, then $x/\norm{x}_1$ is an extreme point of $B_{(X,\norm{\cdot}_1)}$ if and only if $x/\norm{x}_2$ is an extreme point of $B_{(X,\norm{\cdot}_2)}$.
    \item The space $(X,\norm{\cdot}_1)$ is strictly convex (uniformly convex), if and only if $(X,\norm{\cdot}_2)$ is strictly convex (uniformly convex).
    \item If $p,q\in [1,\infty]$ and $n\in\mathbb{N}$ with $n\geq 2$, then the $\ell^p$ and $\ell^q$ norms on $\mathbb{R}^n$ are angularly equivalent, if and only if $p\neq q.$
\end{enumerate}
\end{proposition}

In this paper, we further showcase how angular equivalent norms share other geometrical properties, similar to results (AE3) and (AE4) in Proposition \ref{prop:results}. In Section \ref{section:unsq}, we see that angular equivalence also preserves uniform non-squareness, and in Section \ref{section:exposed}, we also show that angular equivalence preserve exposed points of a unit ball. In \cite{angular}, a counter example is given to the following question: If $X$ is a real normed spaces with two angularly equivalent norms $\norm{\cdot}_1$ and $\norm{\cdot}_2$, are their dual norms $\norm{\cdot}^*_1$ and $\norm{\cdot}^*_2$ equivalent on $X^*$? In Section \ref{section:duality}, extra conditions to the underlying space $X$ are given to obtaint an affirmative answer, namely strict convexity, smoothness, and reflexivity. 

\section{Preliminary}

Let $(X, \norm{\cdot})$ be a normed space. Throughout the paper, we use the standard notation of $S_X$ and $B_X$ for the unit sphere and unit ball, respectively, of the normed space $X$. Let $x_0 \in X$. The one-sided G\^ateaux derivatives 
$$G^{\pm}(x_0,y)=\lim_{t\rightarrow 0^\pm} \frac1t\left(\|x_0+ty\|-\|x_0\|\right)$$
exist for all $y\in X$ \cite[Lemma 5.4.14]{Megginson}. Furthermore, Lemma 5.4.14 of Megginson \cite{Megginson} also gives the result that $G^{\pm}$ is sub-(super-)additive with respect to the second argument, as summarised in the following proposition. 
\begin{proposition}\label{prop:subadditive}
Let $(X, \norm{\cdot})$ be a normed space. For any $x,y,z\in X$, we have
$$G^+(x,y+z)\leq G^+(x,z)+G^+(y,z)$$
and
$$G^-(x,y+z)\geq G^-(x,z)+G^-(y,z).$$
\end{proposition}

Let $(X, \norm{\cdot})$ be a real normed space. For any $x,y\in X$,  
$$g^\pm(x,y)=\|x\|\lim_{t\rightarrow 0^\pm} \frac1t\left(\|x+ty\|-\|x\|\right)=\norm{x}G^\pm(x,y).$$
We recall the following result (see \cite[Lemma 1]{Milicic-convexity}) which readily follows from the definition of the mapping $g^\pm$.
\begin{proposition}\label{prop:useful}
Let $(X,\norm{\cdot})$ be a real normed space. For any $x,y\in X$, we have the following inequality
\[
-\norm{x}\norm{y}\leq \norm{x}(\norm{x}-\norm{x-y})\leq g^-(x,y)\leq g^+(x,y)\leq \norm{x}( \norm{x+y}-\norm{x})\leq \norm{x}\norm{y}.
\]
\end{proposition}

We note that there is a connection between the $g$-functional with the notion of semi-inner product. We recall the definition of semi-inner product. 
\begin{definition}\label{dfn:sip}
Let $X$ be a vector space over the field $\mathbb{K}$. The mapping $[\cdot, \cdot]: X\times X \rightarrow \mathbb{K}$ is called a semi-inner product, if for all $x,y,z\in X$ and $\alpha\in \mathbb{K},$ the following properties are satisfied:
\begin{enumerate}[(S1)]
\item $[x+y,z]=[x,z]+[y,z]$;
\item $[\alpha x, y]=\alpha [x,y]$;
\item $[x,x]\geq 0$ and $[x,x]=0$ implies $x=0$;  
\item $|[x,y]|^2 \leq [x,x] [y,y]$;
\item $[x,\alpha y]= \bar{\alpha}[x,y]$.
\end{enumerate}
\end{definition}
\noindent Lumer \cite{Lumer} introduced this concept without (S5) which was later added by Giles \cite{Giles}. 

\begin{remark}
Let $X$  be a vector space equipped with a semi-inner product $[\cdot,\cdot]$. Then,
$$\|x\|:=[x,x]^{\frac12}, \quad (x\in X),$$
is a norm on $X$ (see \cite[Proposition 3]{Dragomir}). We therefore say that on a normed space $(X,\norm{\cdot})$ with a semi-inner product $[\cdot, \cdot]$, that $[\cdot, \cdot]$ generates the norm $\norm{\cdot}$ if  $\|x\|=[x,x]^{\frac12}$, for all $x\in X.$ We note that such a semi-inner product always exists on a normed space $X$ (\cite[Theorem 1]{Giles}). The next proposition provides a condition for uniqueness.
\end{remark}

\begin{proposition}[Dragomir \cite{Dragomir}, Proposition 4, p. 21]\label{prop:unique}
Let $(X,\norm{\cdot})$ be a normed space. Then $X$ is smooth if and only if there exists a unique semi-inner product which generates $\norm{\cdot}$.
\end{proposition}

Let $(X,\norm{\cdot})$ be a real normed space. Recall that the $g$-functional relative to $\norm{\cdot}$ is the map $g\colon X\times X\to \real$ given by 
$$g(x,y)=\frac12(g^+(x,y)+g^-(x,y)), \quad x,y\in X.$$
We are in a position to specify the construction of a (unique) semi-inner product, using the $g$-functional relative to $\norm{\cdot}$, which generates $\norm{\cdot}$.
\begin{proposition}\label{prop:unique-g}
Let $(X,\norm{\cdot})$ be a real normed space. Define 
$[\cdot,\cdot]\colon X\times X \to \mathbb{R}$ by
$$[y,x]:=g(x,y), \quad \text{for all } x,y\in X.$$
Then,
\begin{enumerate}[(i)]
\item $[\cdot, \cdot]$ satisfies properties (S2)-(S5) of Definition \ref{dfn:sip};  
\item If $X$ is smooth, then $[\cdot,\cdot]$ is the unique semi-inner product on $X\times X.$
\end{enumerate}
\end{proposition}
\begin{proof}
First we note that $g(x,x)=\norm{x}^2$ for all $x\in X$. We omit the proof of {\it (i)}, as the proof for (S2), (S3), and (S5) readily follows from the definition of $g$, and (S4) follows from Proposition \ref{prop:useful}. We prove {\it (ii)}. First, note that since $X$ is assumed to be smooth, then $g^+\equiv g^-$, i.e. $$g(x,y)=\|x\|\lim_{t\rightarrow 0} \frac1t\left(\|x+ty\|-\|x\|\right), \quad \text{for all } x,y\in X.$$
By Proposition \ref{prop:subadditive}, $g^\pm$ is also sub-(super-)additive with respect to the second argument, and thus $$ g^-(x,y)+g^-(x,z)\leq g^-(x,y+z)=g(x,y+z)=g^+(x,y+z)\leq  g^+(x,y)+g^+(x,z)$$ and since $g^+\equiv g^-$, we get equality, and therefore,  $$[y+z,x]=g(x,y+z)=g(x,y)+g(x,z)=[y,x]+[z,x].$$
This shows (S1) of Definition \ref{dfn:sip} and together with {\it (i)}, we conclude that $[\cdot,\cdot]$ is a semi-inner product which generates $\norm{\cdot}$. Uniqueness follows from Proposition \ref{prop:unique}. 
\end{proof}

\begin{example}\label{ex:ell-p}[Mili\v{c}i\'c \cite{Milicic-parallelogram}, p. 72] From Proposition \ref{prop:unique-g}, we note that the smoothness of the normed space implies the linearity of the $g$-functional (in the second argument). Let $x=(x_i), y=(y_i)\in \ell^p$ with $1 < p < \infty$. The functional 
\begin{equation}
[y,x]_{\ell^p}= g_{\ell^p}(x,y) = \left\{
\begin{array}{ll}
\|x\|_{\ell^p}^{2-p}\sum_{i} |x_i|^{p-1}\mathrm{sgn}(x_i)y_i,  & x\neq 0;\\
0, &x=0;
\end{array}\right.
\end{equation}
is the unique semi-inner product on $\ell^p \times \ell^p$.  We note that 
\begin{equation}
g_{\ell^1}(x,y) = \|x\|_{\ell^1}\sum_{i}\mathrm{sgn}(x_i)y_i
\end{equation}
is linear in the second argument, and thus is a semi-inner product on $\ell^1\times \ell^1$, although the space is not smooth. 
\end{example}

\section{Uniform non-squareness} \label{section:unsq}

Let $(X,\norm{\cdot})$ be a normed space. Recall that $X$ is said to be uniformly convex if for all $\varepsilon \in (0,2)$ there exists $\delta \in (0,1)$ such that the following holds:
$$\text{if }x,y\in S_X \text{ with }\norm{x-y}\geq \varepsilon, \text{ then } \norm{\frac{x+y}2}\leq 1-\delta.$$

\noindent The notion of uniform non-squareness is introduced by James \cite{James} as a weaker form of uniform convexity. In particular, James showed that a Banach space is reflexive provided that the unit ball is uniformly non-square and thus it gave a refinement to the implication of reflexivity by uniform convexity, that is, 
\begin{center}
    Uniform convexity \quad $\Rightarrow$ \quad Uniform non squareness \quad  $\Rightarrow$ \quad Reflexivity.
\end{center}
\begin{definition}\label{dfn:unsq}
Let $(X,\norm{\cdot})$ be a normed space. The space $X$ is said to be uniformly non-square if there exists $\delta \in (0,1)$ such that 
$$\text{if }x,y\in S_X \text{ with  }\norm{\frac{x-y}{2}}\geq 1-\delta, \text{ then } \norm{\frac{x+y}2}\leq 1-\delta.$$
\end{definition}

\begin{remark}
\begin{enumerate}
\item Definition \ref{dfn:unsq} is rewritten from its original definition in \cite{James}. 
\item In $\real^2$, if $1<\lambda<\sqrt{2}$, then the norm $\norm{\cdot}_\lambda$ defined by
$$\norm{(x,y)}_\lambda:=\max\big\{(x^2+y^2)^\frac12,\lambda\max\{|x|,|y|\}\big\}, \quad (x,y)\in \real^2,$$
is uniformly non-square but not strictly convex (hence, not uniformly convex). This example is due to Kato and Takahashi \cite[p. 1058]{Kato-Takahashi}.
\end{enumerate}
\end{remark}

Our aim is to show that uniform non-squareness is shared by angularly equivalent norms. We start with two lemmas which provide characterisations of uniform non-squareness using norm angles. We follow the main idea of the proof of Theorem 2.6 of \cite{angular}. 
\begin{lemma}\label{lemma:unsq-delta}
Let $(X,\norm{\cdot})$ be a normed space. Then $X$ is uniformly non-square if and only if there exists $\delta \in (0,1)$ such that the following holds:
$$\text{if }x,y\in S_X \text{ with } \norm{\frac{x-y}2}\geq 1-\delta, \text{ then }  \tan\left(\frac{\theta(x,y)}2\right)\geq \sqrt{\delta}.$$
\end{lemma}
\begin{proof}
Assume that $X$ is uniformly non-square, i.e. there exists $\eta\in (0,1)$ such that
$$\text{if }x,y\in S_X \text{ with } \norm{\frac{x-y}2}\geq 1-\eta, \text{ then } \norm{\frac{x+y}2}\leq 1-\eta.$$
Set $\delta:=\eta.$ Let $x,y\in S_X$ with $\norm{\frac{x-y}2}\geq 1-\delta=1-\eta.$ Since $x,y\in S_X$, we have the following inequality
$$-1\leq 1-\norm{x-y}\leq g(x,y)\leq \norm{x+y} -1\leq 1, $$
from Proposition \ref{prop:useful}. Therefore, we have  $1+g(x,y)\leq 2$ and $1-g(x,y)\geq 2-\norm{x+y}$. Then, 
\begin{align*}
    \tan\left(\frac{\theta(x,y)}2\right) &\geq \sqrt{\frac{1-g(x,y)}{1+g(x,y)}}\\
    &\geq \sqrt{\frac{1-g(x,y)}2}\geq \sqrt{1-\norm{\frac{x+y}2}}\geq\sqrt{\eta}=\sqrt{\delta}.
\end{align*}
Conversely, assume there exists $\eta\in (0,1)$ such that
$$\text{if }x,y\in S_X \text{ with } \norm{\frac{x-y}2}\geq 1-\eta, \text{ then }  \tan\left(\frac{\theta(x,y)}2\right)\geq \sqrt{\eta}.$$
Choose $\delta:=\min\{\frac{\eta}2, \frac{\eta}{1+\eta}\}>0$. Let $x,y\in S_X$ with $\norm{\frac{x-y}2}\geq 1-\delta\geq 1-\eta$, since $\delta\leq \frac{\eta}2< \eta.$ 
If $\norm{x+y}=0$, then $\norm{\frac{x+y}2}=0\leq 1-\delta$. 
We consider the case $\norm{x+y}\neq 0$. Now,
\begin{align*}
    &\norm{(2-\norm{x+y})x-\norm{x+y}\left(\frac{x+y}{\norm{x+y}}-x\right)}\\
    &=\norm{2x-\norm{x+y}x+\norm{x+y}x-x+y}=\norm{x-y}\geq 2(1-\delta). 
\end{align*}
Thus, either 
$$\norm{(2-\norm{x+y})x}\geq 2\delta$$
or 
$$\norm{\norm{x+y}\left(\frac{x+y}{\norm{x+y}}-x\right)}\geq 2(1-\delta)-2\delta=2-4\delta,$$
which follows from the triangle inequality.
In the first case, we have
$$2-\norm{x+y}=\norm{(2-\norm{x+y})x}\geq 2\delta$$
that is
$$\norm{\frac{x+y}2}\leq 1-\delta, $$
and we are done. In the second case, we have
$$\norm{\frac{x+y}{\norm{x+y}}-x}\geq \frac{2-4\delta}{\norm{x+y}}\geq 1-2\delta\geq 1-\eta$$
by our choice of $\delta\leq \frac\eta2.$ Therefore, by our assumption, 
$$\sqrt{\eta}\leq \tan\left(\frac{\theta(x,y)}2\right)=\sqrt{\frac{1-g(\frac{x+y}{\norm{x+y}},x)}{1+g(\frac{x+y}{\norm{x+y}},x)}},$$
and by rearranging we obtain
$$g\left(\frac{x+y}{\norm{x+y}},x\right)\leq\frac{1-\eta}{1+\eta}.$$
By Proposition \ref{prop:useful} with $x+y$ and $x$, we have 
$$\norm{x+y}-1\leq \frac{g(x+y,x)}{\norm{x+y}}$$ 
and thus
\begin{align*}
    \norm{\frac{x+y}2}\leq \frac12\left(1+\frac{g(x+y,x)}{\norm{x+y}}\right)&\leq \frac12\left(1+\frac{1-\eta}{1+\eta}\right)\\
    &=\frac{1}{1+\eta}=1-\frac{\eta}{1+\eta}\leq 1-\delta
\end{align*}
as we choose $\delta\leq \frac{\eta}{1+\eta}.$
This completes the proof. 
\end{proof}

\begin{lemma}\label{lemma:unsq-equivalence}
Let $(X,\norm{\cdot})$ be a normed space. Then the following are equivalent.
\begin{enumerate}[(i)]
    \item $X$ is uniformly nonsquare.
    \item there exists $\delta \in (0,1)$ such that the following holds:
    $$\text{if }x,y\in S_X \text{ with } \norm{\frac{x-y}2}\geq 1-\delta, \text{ then }  \tan\left(\frac{\theta(x,y)}2\right)\geq \sqrt{\delta}.$$
    \item there exists $\varepsilon\in (0,2)$ and $\delta\in (0,1)$ such that the following holds:
    $$\text{if }x,y\in S_X \text{ with }\norm{x-y}\geq \varepsilon, \text{ then }  \tan\left(\frac{\theta(x,y)}2\right)\geq \delta.$$
\end{enumerate}
\end{lemma}
\begin{proof}
The equivalence of {\it (i)} and {\it (ii)} follows from Lemma \ref{lemma:unsq-delta}. We show that {\it (ii)} and {\it (iii)} are equivalent. Assume that there exists $\eta  \in (0,1)$ such that the following holds:
$$\text{if }x,y\in S_X \text{ with } \norm{\frac{x-y}2}\geq 1-\eta, \text{ then }  \tan\left(\frac{\theta(x,y)}2\right)\geq \sqrt{\eta}.$$
Set $\varepsilon:=2(1-\eta)>0$ and $\delta:=\sqrt{n}>0$. Let $x,y\in S_X$ be such that $\norm{x-y}\geq \varepsilon.$ Thus, $\norm{x-y}\geq 2(1-\eta)$, that is $\norm{\frac{x-y}2}\geq 1-\eta.$ By assumption, $$\tan\left(\frac{\theta(x,y)}2\right)\geq \sqrt{\eta}=\delta.$$
Now we assume that there exists $\varepsilon \in (0,2)$ and $\eta \in (0,1)$ such that the following holds:
    $$\text{if }x,y\in S_X \text{ with }\norm{x-y}\geq \varepsilon, \text{ then }  \tan\left(\frac{\theta(x,y)}2\right)\geq \eta.$$
Set $\delta:=\min\{1-\frac{\varepsilon}2,\eta^2\}>0$. Let $x,y\in S_X$ be such that $\norm{\frac{x-y}2}\geq 1-\delta.$ Thus, by our choice of $\delta\leq 1-\frac{\varepsilon}2$, we have
$$\norm{\frac{x-y}2}\geq 1-\delta\geq \frac{\varepsilon}2, \quad \text{and so} \quad \norm{x-y}\geq\varepsilon.$$
By assumption, we have $\tan\left(\frac{\theta(x,y)}2\right)\geq \eta\geq\sqrt{\delta},$
by our choice of $\delta\leq \eta^2.$
\end{proof}

Now we prove our main result of the section. 
\begin{theorem}
Let $X$ be a real normed space with two angularly equivalent norms $\norm{\cdot}_1$ and $\norm{\cdot}_2$. Then $X$ is uniformly non-square with respect to $\norm{\cdot}_1$ if and only if $X$ is uniformly non-square with respect to $\norm{\cdot}_2$.
\end{theorem}
\begin{proof}
We need to only prove one side of the implication, as the other side follows by reversing the roles of $\norm{\cdot}_1$ and $\norm{\cdot}_2$. Let $C>1$ be such that 
$$\tan\left(\frac{\theta_1(x,y)}2\right)\leq C \tan\left(\frac{\theta_2(x,y)}2\right)$$
for all $x,y\in X$, where $\theta_i(x,y)$ is the norm angle from $x$ to $y$ with respect to $\norm{\cdot}_1$. Since angular equivalence implies norm equivalence, let $M,m>0$ be such that 
$$m\norm{x}_1\leq \norm{x}_2\leq M\norm{x}_1, $$
for all $x\in X.$
Let $X$ be uniformly non-square with respect to $\norm{\cdot}_1$. By Lemma \ref{lemma:unsq-equivalence} part {\it (iii)} there exist $\nu,\eta>0$ such that 
$$\text{if }x,y\in S_{(X, \norm{\cdot}_1)} \text{ with }\norm{x-y}_1\geq \nu, \text{ then } \tan\left(\frac{\theta_1(x,y)}{2}\right)\geq \eta.$$
Set $\varepsilon:=2\frac{M\nu}{m}>0$ and $\delta:=\frac{\eta}{C}>0$. Let $x,y\in S_{(X, \norm{\cdot}_2)}$ with
$\norm{x-y}_2\geq \varepsilon. $
Let $\hat{x}=\frac{x}{\norm{x}_1}$ and $\hat{y}=\frac{y}{\norm{y}_ 1}$. Note that $\norm{\hat{x}}_1=1=\norm{\hat{y}}_1$. Also, since $x\in S_{(X, \norm{\cdot}_2)}$, we have $\norm{\hat{x}}_2=\frac{\norm{x}_2}{\norm{x}_1}=\frac{1}{\norm{x}_1}$, and thus 
$$x=\norm{x}_1\hat{x}=\frac{\hat{x}}{\norm{\hat x}_2}, \quad
\text{and similarly,} \quad
y=\frac{\hat y}{\norm{\hat y}_2}.$$
Using Dunkl-Williams inequality, we get 
\begin{align*}
    \varepsilon\leq \norm{x-y}_2
    &\leq \norm{\frac{\hat x}{\norm{\hat x}_2}-\frac{\hat y}{\norm{\hat y}_2}}_2\\
    &\leq \frac{4\norm{\hat x-\hat y}_2}{\norm{\hat x}_2+\norm{\hat y}_2} \leq \frac{4M\norm{\hat x-\hat y}_1}{m\norm{\hat x}_1+m\norm{\hat y}_1}=\frac{2M}{m}\norm{\hat x-\hat y}_1
\end{align*}
Thus, 
$$\norm{\hat x-\hat y}_1\geq \frac{m\varepsilon}{2M}=\nu.$$
Therefore, 
$$\eta \leq \tan\left(\frac{\theta_1(\hat x,\hat y)}{2}\right)=\tan\left(\frac{\theta_1( x, y)}{2}\right)\leq C \tan\left(\frac{\theta_2( x, y)}{2}\right), $$
that is, 
$$ \tan\left(\frac{\theta_2( x, y)}{2}\right)\geq \frac{\eta}{C}=\delta, $$
and this completes the proof. 
\end{proof}

\section{Exposed points}\label{section:exposed}

Our aim in this section is to prove a similar result to that of Proposition 2 part (AE3), by considering exposed points instead of extreme points. First we recall the following definitions. 
\begin{definition}
Let $(X,\norm{\cdot})$ be a real normed space and $A$ be a subset of $X$. A nonzero $f\in X^*$ is a support functional for $A$ if there is an $x_0\in A$ such that $f(x_0)=\sup\{f(x):x\in A\}$, in which case $x_0$ is a support point of $A$, the set $\{x:x\in X,\ f(x)=f(x_0)\}$ is a support hyperplane for $A$ and the functional $f$ and the support hyperplane are both said to support $A$ at $x_0$.
\end{definition}

\begin{remark}
Note that as a consequence of the Hahn-Banach theorem, for any $x\in X$ there exists $f\in S_{X^*}$ such that $f(x)=\norm{x}$. Also, $f\in S_{X^*}$ supports $B_X$ at $x_0\in S_X$ if and only if $f(x_0)=1.$
\end{remark}

\begin{definition}
Let $(X,\norm{\cdot})$ be a real normed space and $C$ be a nonempty closed convex subset of $X$. A point $x\in C$ is said to be an exposed point of $C$ if there is $f\in X^*$ such that $f$ is bounded from above on $C$ and attains its supremum on $C$ at $x$ and only at $x$. In this case we call $f$ an exposing functional of $C$ and exposing $C$ at $x$.
\end{definition}

\begin{remark}
If $x_0$ is an exposed point of a nonempty closed convex subset $C$ of $X$, then it is also an extreme point. The converse is not true. For instance, the point $A$ in Figure 1 is an extreme point that is not an exposed point of the bounded region. 

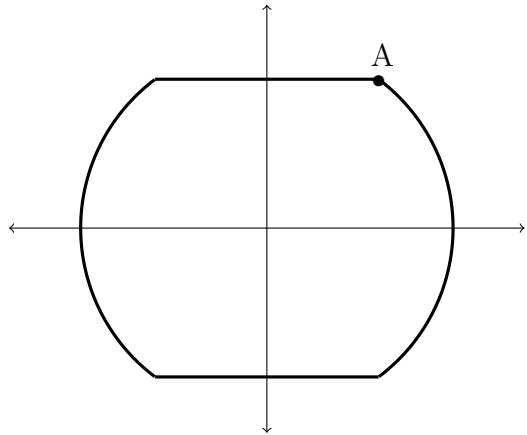
\begin{figure}[h!]
\centering

\begin{tikzpicture}
\begin{axis}[axis equal,
	xmin=-1.1,xmax=1.1,
	ymin=-1.1,ymax=1.2,
	ticks = none,
	xlabel={$\phantom{x}$}, ylabel={$\phantom{y}$},
	]
    \addplot [domain=-0.6:0.6,samples=100,very thick]({x},{-0.8});
	\addplot [domain=-0.6:0.6,samples=100,very thick]({x},{0.8});
	\addplot [domain=-53:53,samples=100,very thick]({cos(x)},{sin(x)});
	\addplot [domain=127:233,samples=100,very thick]({cos(x)},{sin(x)});
	\node at (170,189) {\small{$\bullet$}};
	\node at (172,203) {A};
\end{axis}
\end{tikzpicture}
\caption{An extreme point that is not an exposed point}
\end{figure}
\end{remark}

We recall the following result and refer the readers to Lemma 5.4.16 from Megginson \cite[p. 486]{Megginson} for its proof. We reformulate this for any real normed space.
\begin{proposition}\label{prop:gateaux-ineq}
Let $X$ be a real normed space, $x_0\in S_X$ and $f\in S_{X^*}$. Then $f$ supports $B_X$ at $x_0$ if and only if 
\[
\lim_{t\to0^-}\frac{\norm{x_0+ty}-\norm{x_0}}{t}=G_{-}(x_0,y)\leq f(y)\leq G_{+}(x_0,y)= \lim_{t\to0^+}\frac{\norm{x_0+ty}-\norm{x_0}}{t}
\]
for all $y\in X.$
\end{proposition}
 
We provide a characterisation of an exposed point of the unit ball using the $g$-functional.
\begin{lemma}\label{lemma:exposed-g}
Let $(X,\norm{\cdot})$ be a real normed space. 
Then $x_0\in S_X$ is an exposed point of $B_X$ if and only if $\{y\in S_X: g(x_0,y)=1\}=\{x_0\}$. 
\end{lemma}
\begin{proof}
Let $x_0\in S_X$ be an exposed point of $B_X$ with exposing functional $f$. Thus, $f(x_0)=1$ and  $f(x_0)>f(y)$ for all $y\in S_X$. By Proposition \ref{prop:gateaux-ineq}, we have the following inequality
$$g^-(x_0,y)=G_{-}(x_0,y)\leq f(y)\leq G_+(x_0,y)=g^+(x_0,y),$$
for all $y\in X$. Suppose that there exists  $y_0\in S_X$ with $y_0\neq x_0$ such that $g(x_0,y_0)=1$, i.e. $g^-(x_0,y_0)+g^+(x_0,y_0)=2$. Since $f(y_0)<1,$ by assumption, we have  $g^-(x_0,y_0)\leq f(y_0)<1$ and thus $$g^+(x_0,y_0)=2-g^-(x_0,y_0)>1,$$
contradicting Proposition \ref{prop:useful}. Conversely, assume that $\{y\in S_X: g(x_0,y)=1\}=\{x_0\}$ and suppose that $x_0\in S_X$ is not an exposed point of $B_X$. Thus, if $f\in S_{X^*}$ with $f(x_0)=\norm{x_0}=1$, there exists $y_0\in S_X$ distinct from $x_0$ such that $f(y_0)=\norm{y_0}=1.$ Note that for any $t\in [0,1]$, we have
$$f\left((1-t)x_0+ty_0\right)=tf(x_0)+(1-t)f(y_0)=1.$$
Since $f\in S_{X^*}$, we have $\norm{(1-t)x_0+ty_0}=1$ for all $t\in [0,1]$. Now,
\begin{align*}
    g^{\pm}(x_0,y_0)&=\lim_{t\to0^\pm}\frac1t\left(\norm{x_0+ty_0}-1\right)\\
    &=\lim_{s\to0^\pm}\frac{(1-s)}{s}\left(\norm{x_0+\frac{s}{1-s}y_0}-1\right)\\
     &=\lim_{s\to0^\pm}\frac1{s}\left(\norm{(1-s)x_0+sy_0}-1+s\right)=1.
\end{align*}
Thus, $g(x_0,y_0)=1$ which contradicts the assumption. Therefore, $x_0\in S_X$ must be an exposed point of $B_X$.
\end{proof}

\begin{theorem}
Let $X$ be a real normed space with two angularly equivalent $\norm{\cdot}_1$ and $\norm{\cdot}_2$. Then, $x/\norm{x}_1$ is an exposed point of $B_{(X,\norm{\cdot}_1)}$ if and only if  $x/\norm{x}_2$ is an exposed point of $B_{(X,\norm{\cdot}_2)}$. 
\end{theorem}
\begin{proof}
Let $C>0$ such that 
$$\frac{1-\cos \theta_1(x,y)}{1+\cos \theta_1(x,y)}\leq C\ \frac{1-\cos \theta_2(x,y)}{1+\cos \theta_2(x,y)}$$
for all $x,y\in X$.
It is sufficient to prove one side of the implication as the reverse implication follows from swapping the roles of $\norm{\cdot}_1$ and $\norm{\cdot}_2$. We argue the contrapositive. Assume that $x_0\in S_{(X,\norm{\cdot}_2)}$ is not an exposed point of $B_{(X,\norm{\cdot}_2)}$. By Lemma \ref{lemma:exposed-g}, there exists $y_0\in S_{(X,\norm{\cdot}_2)}$ distinct from $x_0$ such that $g_2(x_0,y_0)=1,$ i.e. $\cos \theta_2(x_0,y_0)=1$ since $x_0,y_0\in S_{(X,\norm{\cdot}_2)}$. Thus, by angular equivalence,
$$\cos \theta_1(x_0,y_0)=1$$
that is, 
$$g_1\left(\frac{x_0}{\norm{x_0}_1},\frac{y_0}{\norm{y_0}_1}\right)=1.$$
By Lemma \ref{lemma:exposed-g} again, since $\frac{x_0}{\norm{x_0}_1}\neq \frac{y_0}{\norm{y_0}_1}$, $\frac{x_0}{\norm{x_0}_1}$ is not an exposed point of $ B_{(X,\norm{\cdot}_1)}$.
\end{proof}

\section{Dual norms}\label{section:duality}

The following theorem is due to Giles \cite[Theorem 6]{Giles}.
\begin{theorem}[Giles, 1967]\label{thm:riesz}
Let $(X, \norm{\cdot})$ be a smooth and uniformly convex Banach space and $[\cdot, \cdot]$ be a semi-inner product which generates $\norm{\cdot}$. Then for all $f\in X^*$, there exists a unique $x\in X$ such that $f(y)=[y,x]$ for all $y\in X.$
\end{theorem}

One of the tools that is used in proving  Theorem \ref{thm:riesz} is that every closed convex subset in a uniformly convex space is a Chebyshev set. Recall that a non-empty subset $A$ of a metric space $(M,d)$ is a Chebyshev set if for every element $x\in M$, there exists exactly one element $y\in A$ such that
$$d(x,y)=d(x,A):=\inf_{z\in A} d(x,z).$$
However, the assumption of uniform convexity may be replaced by a weaker assumption. This result is due to MM Day (cf. \cite[Corollary 5.1.19]{Megginson}): 
\begin{lemma}[Day, 1941]\label{lemma:chebyshev}
If a normed space is strictly convex and reflexive, then each of its nonempty closed convex subsets is a Chebyshev set. 
\end{lemma}
\noindent Recall that uniform convexity implies strict convexity and reflexivity. We prove a version of Theorem \ref{thm:riesz} by replacing uniform convexity with strict convexity and reflexivity and reformulate it in terms of the $g$-functional. We first state some results from  \cite{Dragomir} and \cite{Giles} which are reformulated in terms of the $g$-functional, with the aid of Proposition \ref{prop:unique-g}. Recall that, from Proposition \ref{prop:unique-g}, when $X$ is a smooth normed space, then the $g$-functional gives rise to a unique semi-inner product given by
$$[x,y]=g(y,x),\quad x,y\in X.$$
In a normed space $(X, \norm{\cdot})$ over the field $\mathbb{K}$, $x\in X$ is said to be $B$-orthogonal to $y\in X$ if $\norm{x+\lambda y}\geq \norm{x}$ for all $\lambda\in \mathbb{K}$. In the usual manner, we say that $x\in X$ is $B$-orthogonal to a subset $Y\subseteq X$, if $x$ is $B$-orthogonal to every $y\in Y.$ We restate the following results from \cite{Giles}, in terms of the $g$-functional, instead of a semi-inner product (via Proposition \ref{prop:unique-g}).
\begin{lemma}[Giles \cite{Giles}, Theorem 2]\label{lemma:birkhoff}
If $(X,\norm{\cdot})$ is smooth normed space over $\mathbb{K}$, then $g(x,y)=0$ if and only if $x$ is $B$-orthogonal to $y$.
\end{lemma}

\begin{lemma}[Giles \cite{Giles}, Lemma 5]\label{lemma:sc}
Let $(X,\norm{\cdot})$ be a smooth normed space over reals. Then $X$ is strictly convex if and only if for any nonzero $x,y\in X$, if $g(x,y)=\norm{x}\norm{y}$, then $y=\lambda x$ for some real number $\lambda>0.$ 
\end{lemma}

We now restate Theorem 6 of Giles \cite{Giles} (Theorem \ref{thm:riesz} above) with a weaker assumption of strict convexity and reflexivity in place of uniform convexity. 
\begin{theorem}\label{thm:riesz-g}
Let $(X, \norm{\cdot})$ be a smooth, strictly convex, and reflexive space. Then for all $f\in X^*$, there exists a unique $x\in X$ such that $f(y)=g(x,y)$ for all $y\in X.$ Furthermore, $\norm{f}=\norm{x}$.
\end{theorem}
\begin{proof}
If $f(y)=0$ for all $y\in X$, then we choose $x=0.$ If $f(y)\neq 0$ for some $y\in X$, then the null space of $N$ of $f$ is a proper closed subspace of $X$. Thus, by Lemma \ref{lemma:chebyshev} there exists a unique nonzero vector $z_0\in N$ such that $\norm{y-z_0}=\inf_{z\in N}\norm{y-z}.$ Writing $x_0=y-z_0$, we get
$\norm{x_0}\leq \norm{x_0+z}$
for all $z\in N$, that is $x_0$ is $(B)$-orthogonal to $z$ for all $z\in N$. By Lemma \ref{lemma:birkhoff}, $g(x_0,z)=0$ for all $z\in N.$ We make the following observations: 
\begin{enumerate}[(1)]
    \item If $z_0\in N$, then $f(z)=0=g(x,z_0)$, for any $x=\alpha x_0$ with $\alpha \in \real$. 
    \item Observe that  
    $$f(x_0)=g\left(\frac{f(x_0)}{\norm{x_0}^2}x_0,x_0\right).$$
    So $f(x_0)=g(x,x_0)$ for $x=\frac{f(x_0)}{\norm{x_0}^2}x_0.$
\end{enumerate}
Thus, any $y\in X$ can be written as $y=z_0+ x_0$, where $z_0\in N$, and $0\neq x_0\in X$ is such that $g(x_0, z)=0$ for all $z\in N$. Set $x=\frac{f(x_0)}{\norm{x_0}^2}x_0$. Since $z_0\in N,$ observation (1) gives us $f(z_0)=g(x,z_0)$ and (2) give us 
$f(x_0)=g(x,x_0)$. Therefore, 
\begin{eqnarray*}
f(y)&=&f(z_0+ x_0)\\
&=&f(z_0)+f(x_0)\\
&=&g(x,z_0)+g(x,x_0)=g(x,z_0+ x_0)=g(x,y).
\end{eqnarray*}
To prove uniqueness, let $x,x'\in X$, $x\neq x'$ such that $f(y)=g(x,y)$ and $f(y)=g(x',y)$ for all $y\in X.$ Then, 
$$\norm{x}^2=|g(x,x)|=|g(x',x)|\leq \norm{x'} \norm {x}$$
so $\norm{x}\leq \norm{x'}$ and 
$$\norm{x'}^2=|g(x',x')|=|g(x,x')|\leq \norm{x} \norm {x'}$$
so $\norm{x'}\leq \norm{x}$. Thus, $\norm{x'}= \norm{x}$, and 
$$\norm{x}^2=g(x',x)$$
gives us 
$$\norm{x}\norm{x'}=g(x',x)$$
and so by Lemma \ref{lemma:sc}, we conclude that $x=\lambda x'$. Combining this with $\norm{x'}= \norm{x}$, we conclude that $x=x'.$ Finally, 
$$|f(y)|=|g(x,y)|\leq \|x\|\|y\|$$
and so 
$$\norm{f}=\sup_{0\neq y\in X}\frac{|f(y)|}{\norm{y}}\leq\norm{x},$$
and $$\norm{x}^2=|g(x,x)|=|f(x)|\leq \norm{f}\norm{x}$$
so $\norm{x}\leq \norm{f}.$ This completes the proof. 
\end{proof}

We now restate Theorem 7 of Giles \cite{Giles} in terms of the $g$-functional. 
\begin{corollary}\label{cor:sip-dual-g}
Let $(X,\norm{\cdot})$ be a normed space. Assume that $X$ is smooth, strictly convex, and reflexive. Then, the dual space $X^*$ is smooth, strictly convex, and reflexive; and the $g$-functional on $X^*$, is given by
$$g(\phi,\psi)=g(x_{\psi},x_{\phi}), \quad \text{ for any } \phi,\psi\in X^*,$$
where $x_\phi$ and $x_\psi$ in $X$ are associated to  $\phi$ and $\psi$, respectively, as given in Theorem \ref{thm:riesz-g}.
\end{corollary}
\begin{proof}
By reflexivity of $X$, it follows that $X^*$ is reflexive, and since $X$ is smooth and strictly convex, $X^*$ is smooth and strictly convex. Let $\phi,\psi\in X^*$. By Theorem \ref{thm:riesz-g}, there exist $x_\phi,x_\psi\in X$ such that 
$$\phi(z)=g(x_{\phi},z) \quad\text{and}\quad \psi(z)=g(x_{\psi},z), \quad \text{for all }z\in X,$$
with $\norm{\phi}=\norm{x_\phi}$ and $\norm{\psi}=\norm{x_\psi}$. Define $[\cdot,\cdot]\colon X\times X \to\real$ by
$$[\phi,\psi]:=g(x_\phi,x_\psi),\quad \text{for any }\phi,\psi\in X^*. $$
It is sufficient to show that $[\cdot,\cdot]$ is a semi-inner product on $X^*$, since  smoothness of $X^*$, implies that $[\cdot,\cdot]$ is the unique semi-inner product on $X^*$ which in turn implies that the $g$-functional in $X^*$ is given by 
$$g(\psi,\phi)=[\phi,\psi]=g(x_\phi,x_\psi),\quad \text{for any } \phi,\psi\in X^*,$$
as desired. Let $\phi,\psi,\tau\in X^*$ and $\alpha,\beta\in \real$. Firstly we note the following, 
$$[\phi,\psi]=g(x_\phi,x_\psi)=\phi(x_\psi).$$
Now we show that $[\cdot,\cdot]$ satisfies the properties of semi-inner product. We have
$$[\phi+\psi,\tau]=(\phi+\psi)(x_\tau)=\phi(x_\tau)+\psi(x_\tau)=g(x_{\phi},x_{\tau})+g(x_{\psi},x_{\tau})=[\phi,\tau]+[\psi,\tau].$$
Next, we note that for all $z\in X$,
$$(\alpha\phi)(z)=\alpha\phi(z)=\alpha g(x_{\phi},z)=g(\alpha x_{\phi},z),$$
that is, a one-to-one correspondence between $\alpha\phi\in X^*$ with $\alpha x_{\phi}\in X.$
Thus 
$$[\alpha\phi,\beta\psi]=g(\alpha x_{\phi},\beta x_{\psi})=\alpha\beta g(x_\phi,x_{\psi})=\alpha \beta [\phi,\psi].$$
Next, we have
$$[\phi,\phi]=g(x_\phi,x_\phi)=\norm{x_\phi}^2=\norm{\phi}^2.$$
Thus, $[\phi,\phi]=\norm{\phi}^2\geq 0$ and $[\phi,\phi]=0$ implies $\norm{\phi}^2=0$, so $\norm{\phi}=0.$ 
Finally,
$$|[\phi,\psi]|=|g(x_\phi,x_\psi)|\leq\norm{x_\phi} \norm{x_\psi}=\norm{\phi}\norm{\psi}.$$
This completes the proof.
\end{proof}

\begin{theorem}\label{thm:AE-dual}
Let $X$ be a normed space with two norms $\norm{\cdot}_1$ and $\norm{\cdot}_2$ that are both strictly convex, smooth, and reflexive, and that both norms are angularly equivalent. Then, the dual norms 
$\norm{\cdot}^*_1$ and $\norm{\cdot}^*_2$ are also angularly equivalent. 
\end{theorem}
\begin{proof}
Denote by $g_1$ and $g_2$, the $g$-functional associated to the norm $\norm{\cdot}_1$ and $\norm{\cdot}_2$, respectively. By the assumption of angular equivalence, there exists $C>0$ such that 
$$\frac{1-g_2(x,y)}{1+g_2(x,y)}\leq C \frac{1-g_1(x,y)}{1+g_1(x,y)},$$
for any $x,y\in X$. Take two elements $\phi$ and $\psi$ of the dual space $X^*$. By Theorem \ref{thm:riesz-g}, there exists $x_{\phi}$ and $x_{\psi}$ in $X$ such that
$$\phi(y)=g_1(x_{\phi},y)\quad \text{and}\quad \psi(y)=g_2(x_{\psi},y), \quad \text{for all } y\in X.$$
Thus, we have
\begin{equation}\label{eq:ineq-ae}
\frac{1-g_2(x_{\psi},x_{\phi})}{1+g_2(x_{\psi},x_{\phi})}\leq C \frac{1-g_1(x_{\psi},x_{\phi})}{1+g_1(x_{\psi},x_{\phi})}.
\end{equation}
By Corollary \ref{cor:sip-dual-g}, we have the $g$-functionals on $(X^*,\norm{\cdot}^*_1)$ and $(X^*,\norm{\cdot}^*_2)$, denoted by $g^*_1$ and $g^*_2$, are given by
$$g^*_i(\phi,\psi)=g_i(x_\psi,x_\phi), \quad i=1,2.$$
Consequently, \eqref{eq:ineq-ae} becomes
$$\frac{1-g^*_2(\phi,\psi)}{1+g^*_2(\phi,\psi)}\leq C \frac{1-g^*_1(\phi,\psi)}{1+g^*_1(\phi,\psi)}$$
which shows that the dual norms $\norm{\cdot}^*_1$ and $\norm{\cdot}^*_2$ are also angularly equivalent. \end{proof}

\section{Discussion}
The assumptions of Theorem \ref{thm:AE-dual} are as follows.
\begin{enumerate}[({A}1)]
    \item A real vector space $X$ with two angularly equivalent norms $\norm{\cdot}_1$ and $\norm{\cdot}_2.$
    \item Both $\norm{\cdot}_1$ and $\norm{\cdot}_2$ are strictly convex. 
    \item Both $\norm{\cdot}_1$ and $\norm{\cdot}_2$ are smooth.
    \item Both $\norm{\cdot}_1$ and $\norm{\cdot}_2$ are reflexive. 
\end{enumerate}
Corollary 2.2 of \cite{angular} states that angular equivalence preserves strict convexity and thus (A2) may be weakened to only requiring one of the norms to be strictly convex. This led to the following questions: 
\begin{enumerate}[(Q1)]
    \item Does angular equivalence preserves smoothness? 
    \item Does angular equivalence preserves reflexivity?  
\end{enumerate}

Note also that the statement of Theorem \ref{thm:AE-dual} remains true, when the assumptions (A2)-(A4) are changed to the following.  
\begin{enumerate}[({A}1*)]\setcounter{enumi}{1}
    \item Both $\norm{\cdot}_1$ and $\norm{\cdot}_2$ are uniformly convex. 
    \item Both $\norm{\cdot}_1$ and $\norm{\cdot}_2$ are uniformly smooth. 
\end{enumerate}
By Corollary 2.7 of \cite{angular}, since angular equivalence preserves uniform convexity, (A1*) may be weakened to only requiring that one of the norms to be uniformly convex. This led to the question: 
\begin{enumerate}[(Q1)]\setcounter{enumi}{2}
    \item Does angular equivalence preserves uniform smoothness? 
\end{enumerate}
An affirmative answer to (Q1)-(Q3) will strengthen the result of Theorem \ref{thm:AE-dual}.

\end{document}